\documentclass{amsproc}

\usepackage{amsmath}
\usepackage{amsthm}
\usepackage{amsbsy}
\usepackage{amsopn}
\usepackage{amsmath,amscd}
\usepackage{amssymb}
\usepackage{amsfonts}
\usepackage{multirow}
\usepackage[official ]{eurosym}
\usepackage{xcolor}
\pagecolor{white}

\newtheorem{theorem}{Theorem}[section]
\newtheorem{lemma}[theorem]{Lemma}

\theoremstyle{definition}
\newtheorem{definition}[theorem]{Definition}
\newtheorem{example}[theorem]{Example}

\theoremstyle{remark}
\newtheorem{remark}[theorem]{Remark}

\newcommand{\cH}{\mathcal{H}}

\newcommand{\cT}{\mathcal{T}}

\newcommand{\bN}{\mathbf{N}}

\newcommand{\bR}{\mathbf{R}}

\newcommand{\bZ}{\mathbf{Z}}

\newcommand{\fF}{\mathfrak{F}}

\usepackage{graphicx}

\usepackage{tikz} 
\usetikzlibrary{arrows,decorations.pathmorphing,backgrounds,fit,positioning,shapes.symbols,chains,calc}
\usepackage[all]{xy}
\xyoption{matrix}
\xyoption{arrow}
\usepackage[hmargin=3cm,vmargin=3cm]{geometry}
\usepackage{setspace,kantlipsum}
\tikzset{help lines/.style={step=#1cm,very thin, color=gray},
help lines/.default=.5} 

\usepackage{booktabs}

\newtheorem{proposition}[theorem]{Proposition}
\newtheorem{corollary}[theorem]{Corollary}

\theoremstyle{definition}

\newtheorem{inductive lemma}[theorem]{Inductive Lemma}

\newtheorem{warning}[theorem]{Warning}

\newcommand{\coker}{\mathrm{coker}}
\numberwithin{equation}{section}

\begin{document}

\title{Framed bordism of Lagrangian homotopy spheres via generating functions}


\author{Daniel \'Alvarez-Gavela}
\address{}
\email{dgavela@brandeis.edu}


\subjclass[2020]{Primary }

\date{\today}

\begin{abstract}
In this note we combine a result of B\"okstedt and Waldhausen with the existence theorem for generating functions of tube type for nearby Lagrangian homotopy spheres due to Abouzaid, Courte, Guillermou and Kragh to obtain a restriction on the smooth structure of nearby Lagrangian homotopy spheres. Concretely, it is proved that if a homotopy $n$-sphere $L$ admits a Lagrangian embedding in the cotangent bundle of some other homotopy $n$-sphere $M$, then the difference $[L]-[M]$ in $\theta_n/bP_{n+1}$ is a multiple of the Hopf element $\eta \in \pi^1_s$. In particular it follows that $[L]-[M]$ is 2-torsion in $\theta_n/bP_{n+1}$, hence if $n$ is even then $L\# L$ is diffeomorphic to $M \# M$. As another application, we deduce that if a homotopy $8$-sphere $L$ admits a Lagrangian embedding in $T^*S^8$, then $L$ is diffeomorphic to $S^8$. The results presented in this note are subsumed by a joint work with Abouzaid, Courte and Kragh which treats the general case in which $M$ is an arbitrary smooth manifold. When $M$ is a homotopy sphere the situation is significantly simpler and the purpose of this note is to give a concise exposition of the main result in this special case.
\end{abstract}

\maketitle

 \section{Introduction}
  
 \subsection{Nearby Lagrangian homotopy spheres}\label{sec:laghom}

Let $L \subset T^*S^n$ be an embedded Lagrangian homotopy sphere in the cotangent bundle of the standard sphere. It is not known whether $L$ must be diffeomorphic to the standard sphere. However, there are some known restrictions on the class of $L$ in the finite abelian group $\Theta_n/bP_{n+1}$ of homotopy spheres modulo those which bound a parallelizable manifold, as defined by Kervaire and Milnor in \cite{KM}. 

The first result in this direction was established by Abouzaid \cite{A}, who proved that whenever the dimension $n$ is congruent to $1$ modulo $4$, any Lagrangian homotopy sphere  $L \subset T^*S^n$ must bound a parallelizable manifold. The bordism is constructed from an $(n+1)$-dimensional moduli space of pseudo-holomorphic disks. This technique was further developed by Ekholm and Smith in \cite{ES} and by Ekholm, Kragh and Smith in \cite{EKS}, who in particular proved that the same conclusion holds whenever $n$ is congruent to $3$ modulo $4$. 

Before we state the main result of this note, which is given in Theorem \ref{thm:structural} below, let us formulate some concrete consequences which further constrain the class of $L$ in $ \Theta_n/bP_{n+1}$. 

\begin{corollary}\label{cor:C}
Let $L \subset T^*S^n$ be a Lagrangian homotopy sphere. Then $L \# L$ bounds a parallelizable manifold. \end{corollary}

In other words, the class of $L$ in $\Theta_n/bP_{n+1}$ is 2-torsion. The conclusion of Corollary \ref{cor:C} is nontrivial when the group  $\Theta_n/bP_{n+1}$ is not 2-torsion. According to \cite{KM}, the group $\Theta_n$ of homotopy $n$-spheres (up to h-cobordism) is trivial for $n<7$, and for $7 \leq n \leq 20$  the isomorphism class of the group $\Theta_n/bP_{n+1}$ is tabulated below: \\

\begin{center}
\begin{tabular}{|c|c||c|c|} 
 \hline
$n$ & $\Theta_n/bP_{n+1}$ &$ n$ &  $\Theta_n/bP_{n+1}$   \\ [0.5ex] 
 \hline\hline
7&  $0$& 14 &  $\bZ/2$  \\ 
 \hline
8& $\bZ/2$ & 15 & $\bZ/2$ \\
 \hline
9&  $\bZ/2 \oplus \bZ/2$  & 16 & $\bZ/2$ \\
 \hline
10 & $\bZ/2 \oplus\bZ/3$ & 17 & $\bZ/2 \oplus \bZ/2 \oplus \bZ/2$  \\
 \hline
11  & 0 & 18 & $\bZ/8 \oplus \bZ/2$ \\  
\hline
12  & 0 & 19 &   $\bZ/2$ \\  
\hline
13  & $\bZ/3$ & 20 &  $\bZ/24$  \\   [1ex] 
 \hline
\end{tabular}
\end{center}

\vspace{0.5cm}

From the table above we see that  $\Theta_n/bP_{n+1}$ is certainly not always 2-torsion, with $n=10$ being the smallest dimension in which odd torsion appears. We also recall that for $n$ even we have $bP_{n+1}=0$,  so if $L \subset T^*S^n$ is a Lagrangian homotopy sphere we have that $L\# L$ is trivial in $\Theta_n$ itself and hence by the h-cobordism theorem we deduce:

\begin{corollary}\label{cor:F} Let $L \subset T^*S^n$ be a Lagrangian homotopy sphere of even dimension $n \neq 4$. Then $L \# L$ is diffeomorphic to the standard sphere. \end{corollary}

Going back to the case $n=10$, note that there are 6 different smooth structures on the 10-sphere and Corollary \ref{cor:C} rules out all the possible exotic structures for a Lagrangian homotopy sphere $L \subset T^*S^{10}$ except for one. Another example is $n=18$, where there are 16 different smooth structures and Corollary \ref{cor:C} rules out all but 3 of the exotic structures.  

There are also examples of odd $n$  for which Corollary \ref{cor:C} gives a nontrivial conclusion, the smallest being $n=13$, but of course when $n$ is odd we already know that $L$ itself bounds a parallelizable manifold from \cite{A, ES, EKS}.

In some special dimensions we can show that any homotopy sphere $L \subset T^*S^n$ must itself bound a parallelizable manifold. Let $\pi_n^s$ be the $n$-th stable homotopy group of spheres, $J_n: \pi_n O \to \pi_n^s$ the $J$-homomorphism and $\text{im}(J_n)\subset \pi_n^s$ its image. We denote by $\eta \in \pi_1^s$ the image of the Hopf fibration $S^3 \to S^2$ under the stabilization map $\pi_3(S^2) \to \pi_1^s$. Recall that $\pi_*^s=\bigoplus_i \pi^s_i$ is a ring and consider the map $\pi_{n-1}^s \to \pi_n^s$ given by multiplication with $\eta$.
 \begin{corollary}\label{cor:A}
Let $n$ be such that $\eta \cdot \pi_{n-1}^s \subset \text{im}(J_n)$. Then any Lagrangian homotopy sphere $L \subset T^*S^n$ bounds a parallelizable manifold. If moreover $n$ is even and $\neq 4$, then $L$ is diffeomorphic to $S^n$. 
 \end{corollary}
 
 
Of course the hypothesis of Corollary \ref{cor:A} will not hold in general, and verifying whether it is satisfied for a given $n$ is a nontrivial problem in stable homotopy theory. Note that since $\eta$ is 2-torsion we only need to look at the 2-primary part of the stable homotopy groups of spheres. The relevant computations, which are known for at least $n \leq 60$, are succinctly summarized in the beautiful picture \cite{H} on Allen Hatcher's website. For a reference to any mathematical assertions related to \cite{H} which are made below, the reader may consult Ravenel's book \cite{R}.


For example, when $n=8$ the image of $J_8$  is generated by $\sigma \eta$, which from \cite{H} also generates $\pi^7_s \cdot \eta$, and hence Corollary \ref{cor:A} applies to yield the following conclusion which did not follow from Corollary \ref{cor:C} or \ref{cor:F} since $\Theta_8 \simeq \bZ/2$ is 2-torsion.

\begin{example}\label{ex:n=8}
If $L \subset T^*S^8$ is a Lagrangian homotopy sphere, then $L$ is diffeomorphic to $S^8$.
\end{example}

One may find other examples of $n$ to which Corollary \ref{cor:A} applies by direct inspection of \cite{H}, including the even values $n=14,20,28,30,...$ However, when $n=10$, the image of $J_{10}$ is trivial, whereas the image of $\eta$ generates the 2-primary part of $\pi_{10}^s \simeq \bZ/6$, so Corollary \ref{cor:A} still cannot exclude the remaining possibility for the exotic smooth structure of a Lagrangian homotopy sphere $L \subset T^*S^{10}$.  In particular,
combining Example \ref{ex:n=8} with \cite{A} (which addresses the case $n=9$) it follows that $n=10$ is the smallest dimension for which it is currently not known whether a Lagrangian homotopy sphere $L \subset T^*S^n$ must bound a parallelizable manifold.

\subsection{Main theorem}
Recall the Pontryagin-Thom isomorphism between the group $\Omega_n^{fr}$ of framed bordism classes of closed framed manifolds and the $n$-th stable homotopy group of spheres $\pi_n^s$. Following Kervaire and Milnor \cite{KM}, any homotopy sphere $\Sigma \in \Theta_n$ is stably parallelizable. Further, any two framings of $\Sigma$ differ by an element of $\pi_n O$, so the framed bordism class of $\Sigma$ is well-defined modulo the image of the $J$-homorphism. This leads to a homomorphism
\[\Theta_n\to \coker(J_n),\] 
whose kernel is $bP_{n+1}$. Hence $\Theta_n/bP_{n+1}$ is isomorphic to a subgroup of $\text{coker}(J_n)$. In fact it is isomorphic to either $\text{coker}(J_n)$ or to an index 2 subgroup of $\text{coker}(J_n)$, but we won't need this fact in what follows.
Let $H_n$ denote the image of the composition 
\[\pi_{n-1}^s\overset{\cdot \eta}{\longrightarrow} \pi_n^s \to \coker(J_n).\]
Note that $H_n$ is 2-torsion since $2\eta=0$ in $\pi_1^s \simeq \mathbf Z/2$. The main result of this article is the following.

\begin{theorem}\label{thm:structural}
Let $\Sigma_0$ and $\Sigma_1$ be homotopy spheres such that $\Sigma_0$ admits a Lagrangian embedding into $T^*\Sigma_1$. Then the difference between $\Sigma_0$ and $\Sigma_1$ in the group $\Theta_n/bP_{n+1}$ lies in the subgroup $H_n\subset \coker (J_n)$.
\end{theorem}
 
All the corollaries of Section \ref{sec:laghom} are immediate consequences of Theorem \ref{thm:structural}. 

The question of Lagrangian embeddings is related to the question of whether the symplectic topology of a cotangent bundle remembers the smooth structure of the base. Indeed, a symplectomorphism $T^*\Sigma_0 \simeq T^*\Sigma_1$ yields a Lagrangian embedding $\Sigma_0 \subset T^*\Sigma_1$, namely the image of the zero section. From Theorem \ref{thm:structural} we deduce:

\begin{corollary} Let $\Sigma_0$ and $\Sigma_1$ be homotopy spheres such that $T^*\Sigma_0$ is symplectomorphic to $T^*\Sigma_1$. Then the difference between $\Sigma_0$ and $\Sigma_1$ in the group $\Theta_n/bP_{n+1} $ lies in the subgroup $H_n\subset \text{coker}(J_n)$. 
 \end{corollary}

A completely different proof of Theorem \ref{thm:structural} (and therefore of all the above corollaries) was obtained by Porcelli and Smith in \cite{PS} by combining the results of \cite{ACGK} with Floer homotopy theory.

 \subsection{Idea of the proof}

The proof of Theorem \ref{thm:structural} uses generating functions of tube type, or said differently, generating functions on tube bundles. A tube, as introduced by Waldhausen in \cite{W}, is a codimension zero submanifold with boundary of Euclidean space which is the image under a compactly supported isotopy of the addition of a standard handle to a half-space. We rely on two crucial ingredients:

\begin{itemize}
\item[(1)] Results of Waldhausen and B\"okstedt \cite{W, B, BW} regarding a map on the space of tubes called {\it the derivative}, which is closely related to the splitting map of algebraic K-theory.
\item[(2)] The result \cite{ACGK} of Abouzaid, Courte, Kragh and Guillermou on the existence of generating functions of tube type for nearby Lagrangian homotopy spheres. \end{itemize}

A sketch of the proof is as follows. Waldhausen defined a derivative map $\cT_\infty \to G$, where $\cT_\infty$ is the space of stable tubes and $G$ is the space of stable homotopy equivalences of spheres. By (1) this map factors through the classifying space for spherical fibrations $BG$, where the map $BG \to G$ is the $\eta$-map, which acts on $\pi_n$ by multiplication with the Hopf element when $n \geq 3$. 

Now suppose that $\Sigma_0$ and $\Sigma_1$ are homotopy spheres and $\Sigma_0 \subset J^1\Sigma_1$ is a Legendrian embedding which is generated by the function $f:W \to \bR$ on a tube bundle $W \to \Sigma_1$. We will identify the pushforward under the Waldhausen map of the element in $\pi_n\cT_\infty$ classifying $W$ to be the element of $\pi_nG \simeq \pi_n^s$ whose coset in $\text{coker}(J_n)$ is the difference of $\Sigma_0$ and $\Sigma_1$ in $\Theta_n/bP_{n+1}$. 

The proof of Theorem \ref{thm:structural} is then completed using the existence result for generating functions (2), which states that for $\Sigma_0,\Sigma_1$ homotopy spheres any Lagrangian embedding $\Sigma_0 \subset T^*\Sigma_1$ admits a generating function on some tube bundle.




 \subsection{Beyond homotopy spheres}

Now let $L$ and $M$ be arbitrary closed, connected, smooth manifolds of the same dimension. An exact Lagrangian embedding $L \hookrightarrow T^*M$ is called a {\em nearby Lagrangian}. The well-known {\em nearby Lagrangian conjecture} predicts that any nearby Lagrangian is Hamiltonian isotopic to the zero-section, and in particular that $L$ must be diffeomorphic to $M$.  In the joint work  \cite{AACK} with M. Abouzaid, S. Courte and T. Kragh, the author recently established a general constraint on the smooth structure of nearby Lagrangians which generalizes Theorem \ref{thm:structural}. 
 
We first recall that although it is not known whether the conjecture or the implication on the diffeomorphism type of nearby Lagrangians holds, it is known that the map $L \to M$ obtained by post-composition of a nearby Lagrangian $L \hookrightarrow T^*M$ with the projection $T^*M \to M$ is a homotopy equivalence \cite{KA}.  Surgery theory divides the problem of whether a homotopy equivalence between smooth closed manifolds is  homotopic to a diffeomorphism into two sub-problems, first whether the normal invariant is trivial, and second whether an $L$-group obstruction vanishes. We have nothing to say on the $L$-group obstruction, but are able to constrain the normal invariant of the homotopy equivalence $L \to M$ associated to a nearby Lagrangian $L \hookrightarrow T^*M$. 

The normal invariant of a homotopy equivalence $L \to M$ is a map $M \to G/O$ which is well-defined up to homotopy, where here and below we use the notation $G=\lim_n G_n$ for $G_n$ the space of maps $S^n \to S^n$ of degree $\pm 1$,  $O=\lim_n O_n$ is the stable orthogonal group, and $G/O$ is the homotopy fiber of the map $BO \to BG$ induced by the $J$-map $O \to G$. The precise statement of the constraint is as follows, where we also denote $\mathcal{Q}=\lim_n \mathcal{Q}_n$ for $\mathcal{Q}_n$ the space of non-degenerate quadratic forms on $\bR^n$, and denote by $\mathcal{T}$ the space of tubes.
 
 \begin{theorem}[\cite{AACK}]\label{thm:general}
If $L \hookrightarrow T^*M$ is a nearby Lagrangian, then the normal invariant $M \to G/O$ of the homotopy equivalence $L \to M$ factors as a composition
 $$ M \to B(\mathcal{T}, \mathcal{Q} ) \to B(G/O) \to G/O $$
 where the first map classifies a twisted generating function of tube type for $L$, the second map is a twisted version of the forgetful map $\mathcal{T} \to \mathcal{T}^{\text{TOP} } \simeq \bZ \times BG$ from the space of tubes to the space of topological tubes, and the third map is a twisted version of the S-duality map $BG \to G$.
 \end{theorem}
 
 The above statement is somewhat technical and the proof is much more so. However, in the special case in which $M$ (and therefore also $L$) is a homotopy sphere the situation is considerably simpler. The value of the present note compared to \cite{AACK} is to give an exposition of this result and its proof in this special case, which we hope will be helpful to understand the general idea and may serve as a stepping stone towards \cite{AACK}. 
 
 In both the special and the general case the strategy of proof relies crucially on the foundational work \cite{ACGK} by M. Abouzaid, S. Courte, S. Guillermou and T. Kragh, which establishes existence of twisted generating functions of tube type for nearby Lagrangians, generalizing the breakthrough result \cite{K} of T. Kragh in $\mathbf{R}^{2n}$. However, for homotopy spheres the situation is significantly simpler in that \cite{ACGK} shows that nearby Lagrangian homotopy spheres admit genuine, i.e. untwisted, generating functions. As a result, much of the technical work in \cite{AACK} may be avoided if one is willing to appeal to the article \cite{BW} of B\"okstedt and Waldhausen. 
 
 Unfortunately, the authors of \cite{AACK} were unable to fully parse the details of \cite{BW}. Moreover, the scope of the discussion in \cite{BW} is insufficient for the proof of Theorem \ref{thm:general} in full generality. For this reason, \cite{AACK} includes a self-contained treatment of the relevant results of \cite{BW}, and moreover upgrades them to the category of topological monoids, which is necessary to account for the twisting of twisted generating functions. In the present expository article we restrict the discussion to the special case of homotopy spheres, where the twisting is unnecessary, and directly appeal to the results of \cite{BW}, treating them as a black box.
 
As mentioned above, a different proof of the main result of this note concerning Lagrangian homotopy spheres, Theorem \ref{thm:structural}, was obtained by Smith and Porcelli in \cite{PS} with completely different methods. In the more recent work \cite{PS2} Smith and Porcelli were able to push these methods further to obtain a general result very close to Theorem \ref{thm:general}, namely that for $L$ and $M$ arbitrary smooth, closed, connected manifolds, and for $L \hookrightarrow T^*M$ an exact Lagrangian embedding, the normal invariant $M \to G/O$ of the homotopy equivalence $L \to M$ factors as a composition $M \to B(G/O) \to G/O$. Moreover, they identified the map $B(G/O) \to G/O$ as an $\eta$-map, i.e. the map induced by the action of the generator $\eta \in \pi_1^s \simeq \bZ/2$ with respect to the infinite loopspace structure on $G/O$ it inherits as the fiber of the delooped $J$-map $BO \to BG$, which is a map of infinite loopspaces with respect to the operation of direct sum on $BO$ and the operation of smash product on $BG$. 

Theorem \ref{thm:general} also provides a factoring $M \to B(G/O) \to G/O$, but does not identify the map $B(G/O) \to G/O$ as the $\eta$-map, merely as the map induced by the S-duality map $BG \to G$, which is checked explicitly in \cite{AACK} to be 2-torsion. However, Theorem \ref{thm:general} provides the further factoring through the map $B(\mathcal{T},\mathcal{Q}) \to B(G/O)$, which \cite{PS2} does not. It would be interesting to compare the factorings of \cite{AACK} and \cite{PS2}, and in particular to verify that the map $B(G/O) \to G/O$ in Theorem \ref{thm:general} is indeed the $\eta$-map. 

A prior version of this note was originally uploaded to the author's website in 2020 as a placeholder for \cite{AACK} but is only now being uploaded to the arXiv, with minor modifications, and will appear in the proceedings of the 2025 Georgia International Topology Conference.

\subsection{Acknowledgements}
The author first started discussing these ideas with Sylvain Courte after attending a talk he gave at Jussieu in which he presented the main results of \cite{ACGK}. A visit to the Institut Fourier in Grenoble was then arranged, during which the ideas in this note emerged. The author is grateful to Sylvain and to the Institut Fourier for their hospitality. The author is also grateful to Mohammed Abouzaid and Thomas Kragh, who had also been thinking about the normal invariant of nearby Lagrangians and welcomed a four way collaboration to develop these ideas further, culminating in the article \cite{AACK}. The author is also grateful to Kiyoshi Igusa, from whom he has learned a great deal of algebraic K-theory, as well as to S\o ren Galatius and Sander Kupers for many helpful conversations on homotopy theory. Finally, the author is grateful to the Institute of Advanced Study, Princeton University, MIT and Brandeis University, where the author was appointed in various capacities during the period in which these ideas were developed, and gratefully acknowledges the support of the NSF under grant DMS-2203455, the Simons Foundation, and the K-Theory Foundation. 

\section{Bundles of tubes}

\subsection{Pontryagin-Thom on tubes}

Let $E$ be a $k$-dimensional linear subspace of $\bR^N$. Consider the codimension zero submanifold with boundary $T_E \subset \bR^{N+1}$ obtained by attaching to the half-space $\{x_{N+1} \leq 0 \}$ a standard $(N+1)$-dimensional index $k$ handle along the unit sphere of $E \subset \{x_{N+1}=0\}$. More concretely, $T_E$ is the union in $\bR^{N+1}$ of the half-space $\{x_{N+1} \leq 0\}$ and an $\varepsilon$-neighborhood of the $k$-dimensional unit sphere of the $(k+1)$-dimensional subspace of $\bR^{N+1}$ spanned by $E \oplus 0 \subset \bR^{N+1}$ and $\partial / \partial x_{N+1}$, suitably smoothed so that $T_E$ is a smooth codimension zero submanifold with boundary. We call such a codimension zero submanifold $T_E \subset \bR^{N+1}$ a {\it rigid tube}. Since the thickening and smoothing are homotopically canonical, for fixed $k$ and $N$ the space of rigid tubes $T_E$ as above has the homotopy type of $BO(k,N)$, the Grassmannian of real $k$-dimensional linear subspaces of $\bR^N$.

\begin{definition} A {\it tube} is any codimension zero submanifold $T \subset \bR^{N+1}$ which is the image of a rigid tube under a smooth isotopy of $\bR^{N+1}$, fixed outside of a compact set.  \end{definition}

\begin{warning}
The definition of a tube given above is almost identical to the definition of a tube given by Waldhausen in \cite{W} in terms of partitions, but is slightly different from the notion of tube considered in \cite{AACK}, where we chose to work with a different model which is better suited for a treatment of the monoidal structure. The difference is only unstable, i.e. the two notions of a tube yield the same space of stable tubes. 
\end{warning}

\begin{definition} Let $M$ be a closed manifold. A {\em tube bundle} $W \to M$ is a smooth fibre bundle of manifolds whose fibres are tubes $T \subset \bR^{N+1}$ in a fixed Euclidean space. \end{definition}

For $T \subset \bR^{N+1}$ a tube, we consider vector fields $X_T$ on $T$ such that:
\begin{itemize}
\item[(1)] $X_T$ is outwards pointing along $\partial T$.
\item[(2)] $X_T= \partial_{N+1}$ outside of a compact subset.
\end{itemize}

Let $W \to M$ be a tube bundle over a closed smooth manifold $M$. We consider fibrewise vector fields $X:W \to \bR^{N+1}$ such that the restriction of $X$ to each fibre is a vector field satisfying (1) \& (2) and such that
\begin{itemize}
\item[(3)]$X$ has 0 as a regular value.
\end{itemize}
If $\fF$ is a framing of $M$, i.e. a trivialization of its stable normal bundle, then the manifold $X^{-1}(0)$ gets an induced framing $\fF_X$. Explicitly, one has a framing of the normal bundle of $X^{-1}(0)$ in the total space of the tube bundle $W$ given by the pullback $(dX)|_{X=0}^{-1} \mathcal{B}$ of the canonical basis $\mathcal{B}=(e_1,\ldots,e_{N+1})$ of $\bR^{N+1}$, viewed as a basis for $T_0\bR^{N+1}$. But $W$ is a codimension zero submanifold of $M \times \bR^{N+1}$, so the stable normal bundle of $X^{-1}(0)$ is the Whitney sum of the normal bundle to $X^{-1}(0)$ in $W$ and the pullback to $X^{-1}(0)$ of the stable normal bundle of $M$. Therefore the framings $(dX)|_{X=0}^{-1} \mathcal{B}$ and $\fF$ combine to yield a framing $\fF_X$ of the stable normal bundle of $X^{-1}(0)$. 

\begin{lemma}\label{lem:indep}
Let $W \to M$ be a tube bundle over a framed manifold $(M ,\fF)$. For $X$ satisfying (1), (2) \& (3) the framed bordism class of $(X^{-1}(0), \fF_X)$ is independent of $X$.
\end{lemma}

\begin{proof}
The space of all $X$ satisfying (1) and (2) is convex. The straight line homotopy $X_t=(1-t)X_0+tX_1$ between any two $X_0$ and $X_1$ satisfying (1), (2) and (3) will therefore automatically satisfy (1) and (2). Furthermore, by Thom transversality the homotopy $X_t$ may be perturbed, keeping it fixed near $t=0,1$, as well as keeping it fixed outside of a compact subset for all $t \in [0,1]$, so that $0$ is a regular value of the map $X:W \times [0,1] \to \bR^{N+1}$, $X(\cdot,t)=X_t(\cdot)$. Then $X^{-1}(0)$ yields a framed bordism between $X_0^{-1}(0)$ and $X_1^{-1}(0)$.
\end{proof}

\begin{definition} Let $W \to M$ be a tube bundle over a framed manifold $(M, \mathfrak{F})$. We denote by $\beta(W,\fF) \in \Omega^{fr}_n$ the framed bordism class $(X^{-1}(0),\fF_X)$ for any $X$ satisfying (1), (2) \& (3). \end{definition}


\subsection{Generating functions on tubes}

We review the generating function construction in the setting of tube bundles.

Let $T \subset \bR^{N+1}$ be a tube. We consider functions $g:T \to \bR$ such that:
\begin{itemize}
\item[(1)] $\partial T$ is a regular level set of $g$.
\item[(2)] $g=x_{N+1}$ outside of a compact set.
\end{itemize}

Let $W \to M$ be a tube bundle. We consider functions $f:W \to \bR$ such that the restriction of $f$ to each fibre is a function satisfying (1) \& (2) and such that:
\begin{itemize}
\item[(3)] the fibrewise Euclidean gradient $\nabla_Tf:W \to \bR^{N+1}$ has 0 as a regular value.
\end{itemize}

We denote by $f_m:T_m \to \bR$ the restriction of $f$ to the fibre over $m \in M$. 
Recall that $J^1M=T^*M \times \bR$ is equipped with the canonical contact form $\alpha=dz-pdq$ for $pdq$ the Liouville form on $T^*M$. A Legendrian submanifold $\Lambda \subset J^1M$ is a smooth submanifold of the same dimension as $M$ such that $\alpha|_\Lambda=0$. 

\begin{definition} A {\it generating function of tube type} for a Legendrian $\Lambda \subset J^1M$ consists of a function $f:W \to \bR$ on a tube bundle $W \to M$ which satisfies the properties (1), (2) \& (3) and which is a generating function for $\Lambda$ in the usual sense (see for example \cite{G}).  \end{definition}

Recall that $T^*M$ is equipped with the symplectic form $\omega=dp \land dq$. A Lagrangian submanifold $L \subset T^*M$ is a smooth submanifold of the same dimension as $M$ such that $\omega|_M=0$. Since $dp \land dq = d(pdq)$, an embedded Legendrian submanifold $\Lambda \subset J^1M$ generically projects to an immersed Lagrangian submanifold $L \subset T^*M$. In this case we say that $\Lambda$ is a Legendrian lift of $L$. If $L \subset T^*M$ is an exact Lagrangian submanifold, which means that $pdq|_L$ is an exact form, then $L$ admits a Legendrian lift. If $L$ is connected then the Legendrian lift of an exact Lagrangian embedding $L \subset T^*M$ is unique up to translations in the $\bR$ factor of $J^1M=T^*M \times \bR$. 

We say that a function generates an exact Lagrangian $L \subset T^*M$ if it generates a Legendrian lift of $L$. Note that a generic function $f:W \to \bR$ generates an embedded Legendrian, but its Lagrangian projection will in general only be immersed. 

Note also that if $\Lambda \subset J^1M$ is a Legendrian admitting a generating function $f$ of tube type on a tube bundle $W \to M$ and $M$ is framed, then $\Lambda$ is naturally equipped with an induced framing, see \cite{G}. More precisely, to a framing $\mathfrak{F}$ of $M$ and a generating function $f:W \to \bR$ for $\Lambda$ on a tube bundle over $M$, we assign a framing $\fF_f$ of $\Lambda$.

\begin{definition} For $\fF$ a framing of $M$ and $\Lambda \subset J^1M$ a Legendrian generated by a function $f:W \to \bR$, we call $\fF_f$ the {\it induced framing} on $\Lambda$. \end{definition}

\begin{lemma}\label{lem:rep}
Let $(M, \mathfrak{F})$ be a framed manifold and $\Lambda \subset J^1M$ a Legendrian submanifold admitting a generating function $f:W \to \bR$ on a tube bundle $W \to M$. Then $(\Lambda, \fF_f)$ is equal to $\beta(W,\fF)$ in $\Omega^{fr}_n$.
\end{lemma}

\begin{proof} We can use the fiberwise Euclidean gradient $\nabla _Tf$ as a fibrewise vector field to get a representative for the framed bordism class. Then by definition of generating function the fiberwise critical set $\Sigma = \nabla_T f^{-1}(0)$ comes equipped with a diffeomorphism to $\Lambda$, and under this identification the induced framings $\fF_f$ and $\fF_{\nabla_T f}$ coincide.\end{proof}

\subsection{Existence of generating functions of tube type}

Let $L$ be the zero section in $T^*S^n$. Then $L$ admits a generating function $f$ over any bundle of rigid tubes over $S^n$. Indeed, up to contractible choice, a rigid tube bundle $T_E$ is completely determined by a vector bundle $E$ over $S^n$. We can then take $f$ to be a standard function with one unique critical point on each fibre. Near the critical point this is given by a family of quadratic forms $Q_E$ over $S^n$ with negative eigenspace $E$, i.e. classified by the same map $S^n \to BO$. 

By the homotopy lifting property for generating functions under Hamiltonian isotopies, the nearby Lagrangian conjecture predicts that if $\Sigma_0$ and $\Sigma_1$ are homotopy spheres, then any Lagrangian embedding $\Sigma_0 \subset T^*\Sigma_1$ admits a generating function on any bundle of rigid tubes, including the trivial bundle. We note that a generating function of tube type on a trivial (rigid) tube bundle is essentially the same thing as the more classical notion of a generating function quadratic at infinity, i.e. a generating function  $g:M \times \bR^N \to \bR$ on the trivial $\bR^N$ bundle which outside of a compact subset is equal to a fixed non-degenerate quadratic form $q:\bR^N \to \bR$. While the problem of existence of such generating functions seems currently out of reach, if we allow for potentially non-rigid tube bundles we have the following result, which will be a key ingredient to our discussion in what follows.

\begin{theorem}[\cite{ACGK}]\label{thm:existence} If $\Sigma_0, \Sigma_1$ are homotopy spheres, then any Lagrangian embedding $\Sigma_0 \subset T^*\Sigma_1$ admits a generating function of tube type on some tube bundle $ W \to \Sigma_1$. \end{theorem}

\begin{remark}
Some minor translation needs to be carried out in order to go from the notion of generating function of tube type used in \cite{ACGK} to the notion of generating function of tube type defined above, but this is straightforward. Less trivial is the translation to the notion of generating function of tube type used in \cite{AACK}, indeed this translation takes up an entire section of \cite{AACK}. 
\end{remark}

For $M$ an arbitrary smooth closed manifold and for $L \subset T^*M$ an arbitrary nearby Lagrangian, the problem of existence of generating functions of tube type is currently wide open. In  \cite{ACGK} it is proved that this problem is equivalent to asking whether the stable Gauss map $L \to U/O$ is trivial, and moreover in \cite{ACGK} it is proved that without any hypothesis any nearby Lagrangian admits a so-called {\it twisted generating function} of tube type (on a {\em twisted tube bundle}). We will not need this notion in the present note and instead direct the reader who is interested in going beyond homotopy spheres to the article \cite{AACK}.

\section{The space of tubes}

\subsection{Relation to algebraic K-theory}

Recall that for $E$ a $k$-dimensional linear subspace of $\bR^N$ we have a rigid tube $T_E \subset \bR^{N+1}$ obtained by attaching a standard handle along the unit sphere of $E \subset \{x_{N+1}=0\}$. Let $\cT_{k,N}$ be the space of all tubes which are isotopic to $T_E$ by an isotopy fixed outside of a compact set, called $(k,N)$-tubes. 

One could equip $\cT_{k,N}$ with a topology directly, or one could define $\cT_{k,N}$ as the geometric realization of a simplicial space as in \cite{W}, or one could give a model for $\cT_{k,N}$ in terms of a space of suitable functions on $\bR^{N+1}$ for which tubes are regular sub-level sets as in \cite{AACK}, but here we will be agnostic about the details as they will not be relevant for the purposes of this expository note. 

There are two stabilizations procedures $\cT_{k,N} \to \cT_{k,N+1}$ and $\cT_{k,N} \to \cT_{k+1,N+1}$ described in \cite{W}, where we note that the ambient dimension $N$ increases by one in both but the index $k$ of the handle stays fixed in the first and increases by one in the second. In \cite{AACK} concrete models for these stabilizations are given as part of a richer monoid structure on the disjoint union $\coprod_{k,N}  \cT_{k,N}$, however we will not need to be precise for the present exposition and will only need some formal properties satisfied by these stabilization procedures which we may take as a black box.

If we first take the direct limit with respect to the stabilizations $\cT_{k,N} \to \cT_{k,N+1}$ we obtain a space $\cT_k$, which informally may be thought of as the space of tubes of index $k$ in which the ambient dimension $N$ may be taken to be as large as desired. The stabilizations $\cT_{k,N} \to \cT_{k+1,N+1}$ then descend to stabilizations $\cT_k \to \cT_{k+1}$, the direct limit of which is denoted by $\cT_\infty$ and called the {\em space of stable tubes}. Without this final stabilization, the disjoint union $\cT=\coprod_k \cT_k$ may be simply called the {\em space of tubes} (note that $\pi_0 \cT \simeq \bN$, given by the index $k$). In this note we will focus on the stable space $\cT_\infty$. The (stable and unstable) spaces of tubes were  introduced by Waldhausen in his manifold  approach to the algebraic K-theory of spaces \cite{W}.

Following \cite{W}, we also consider the space $BG$, which is the classifying space of stable spherical fibrations. Here we recall $G_n$ consists of maps $S^n \to S^n$ of degree $\pm 1$ and $G=\lim_n G_n$ under suspension. The relation between the space of tubes and Waldhausen's algebraic K-theory of spaces may be phrased in terms of a homotopy cartesian square:
\[ \xymatrix{
\cT_\infty \ar[d] \ar[r] & BG \ar[d]\\
\Omega^\infty S^\infty  \ar[r]           & A(\ast)   } \]
The map $\cT_\infty \to BG$ is the map which to a bundle of tubes assigns the underlying spherical fibration. Waldhausen showed that $BG$ can also be thought of as the space of stable tubes in the topological category, so that the map $\cT_\infty \to BG$ is simply the forgetful map from the smooth category to the topological category. The map $\cT_\infty \to \Omega^\infty S^\infty$ is the {\em Waldhausen derivative}, which we will define in Section \ref{sec: wald der}.

The assignment $E \to T_E$ yields a map $BO(k,N) \to \cT_{k,N}$ which stabilizes to a map $BO(k) \to \cT_k$, which in turn stabilizes to a map $BO \to \cT_\infty$ known as the  {\it rigid tube map}, and which was first studied by Goodwillie and Waldhausen. Let $O$ be the stable orthogonal group, $G$ the space of stable homotopy equivalences of spheres, $J:O \to G$ the $J$-homomorphism and $\cH_\infty$ the space of stable h-cobordisms of a point (i.e. the space of h-cobordisms of an $n$-disk suitably stabilized in $n$).  It was also shown by Waldhausen \cite{W} that the rigid tube map fits into a commutative diagram where the columns are fibration sequences:
\[ \xymatrix{
G/O \ar[d] \ar[r] & \cH_\infty \ar[d] \\
BO \ar[r] \ar[d]          & \cT_\infty \ar[d]  \\
BG \ar[r]^= & BG} \]

The first column comes from the $J$-homomorphism. The map $\cH_\infty \to \cT_\infty$ consists of adding an h-cobordism on top (or to the side) of a rigid tube. The map $G/O \to \cH_\infty$ is the Hatcher map, which is a rational homotopy equivalence by a theorem of B\"okstedt \cite{B}. The bottom map is the identity, hence the rigid tube map $BO \to \cT_\infty$ is also a rational homotopy equivalence. This fact was used in an essential way in \cite{ACGK} for the proof of Theorem \ref{thm:existence}, which we rely on for our applications. 

We also mention here that further study of the fibration sequence $\cH_\infty \to \cT_\infty \to BG$ and an application of the plus construction lead to a proof of Waldhausen's {\em parametrized stable h-cobordism theorem}, see \cite{W}. We will not need this result for the purposes of the present article, but we mention it since we believe it may be possible to extract further mileage in symplectic topology from this mostly unexplored connection.

\subsection{The classifying space of tube bundles}

The crucial property from \cite{W} which we will need for our applications is that $\cT_\infty = \lim_{k,N} \cT_{k,N}$ is the classifying space for stable tube bundles over manifolds. We explain the significance of this statement.

First, $\cT_{k,N}$ is the classifying space for $(k,N)$-tube bundles over compact manifolds. Explicitly, this means that for $M$ compact there is a one-to-one correspondence between homotopy classes of maps $M \to \cT_{k,N}$ and $(k,N)$-tube bundles $W \to M$ up to equivalence, where two tube bundles $W_0$ and $W_1$ over a manifold $M$ are said to be equivalent if there exists a fibrewise diffeomorphism:
\[
\xymatrixcolsep{5pc}\xymatrix{W_0 \ar[rr] \ar[dr] & & W_1 \ar[dl]  \\  & M  &  
} \]

Hence for every $(k,N)$-tube bundle $W \to M$ we get a well-defined homotopy class of maps $M \to \cT_{k,N}$. After stabilizing, we deduce that to every tube bundle $W \to M$ is assigned a class in $[M, \cT_\infty]$ which classifies $W$ up to {\it stable equivalence}, i.e. up to equivalence after arbitrary stabilization. Conversely, every homotopy class of maps $M \to \cT_\infty$ can be realized by a tube bundle which is unique up to stable equivalence. 

For $W \to M$ a tube bundle and $\fF$ a framing of $M$, it is straightforward to check that the framed bordism class $\beta(W, \fF)$ only depends on the stable equivalence class of $W$. Thus we may think of $\beta(\cdot , \fF)$ as a map $[M, \cT_\infty ] \to \Omega^{\text{fr}}_n$. For simplicity, consider the case where $M$ is a homotopy sphere $\Sigma$. Up to homotopy there is a unique orientation preserving homotopy equivalence $S^n \to \Sigma$, so we may safely identify $[\Sigma, \cT_\infty]$ with $\pi_n \cT_\infty$. Hence to a tube bundle $W \to \Sigma$ is assigned a class $\alpha \in \pi_n \cT_\infty$, which we call the {\it classifying element}. For $\fF$ a fixed framing of $\Sigma$, the resulting map $\pi_n \cT_\infty \to \Omega^{fr}_n$ is easily seen to be a group homomorphism. Indeed, we may always assume our tubes and fibrewise vector fields to be standard near a point.

\subsection{Preliminary computations in framed bordism}\label{sec:prel}

For some classes in $\pi_n \cT_\infty$ the associated tube bundles are not so interesting from the viewpoint of framed bordism. For example, let $\gamma \in \pi_n BO$. Given a framing $\fF$ of a homotopy sphere $\Sigma$ we denote by $\fF_\gamma$ the framing $\fF$ twisted by the image of $\gamma$ under the homomorphism $\pi_nBO \to \pi_nO$ induced by the map $BO \to O$ which comes from the fibration $O \to U \to U/O$ after applying Bott periodicity to $\Omega (U/O)  \to O$ as in \cite{G}. This map can be described more explicitly as follows. Recall that $BO(k,N)$ is the space of $k$-dimensional linear subspaces of $\bR^N$. To each $E \in BO(k,N)$ we assign the rigid motion of $\bR^N$ given by reflection on $E$. This is an element of $O(N)$. Stabilizing we get a map $BO \to O$.


\begin{lemma}\label{prop:BO} Let $\Sigma$ be a homotopy sphere and let $\gamma \in \pi_nBO$. Denote by $W_\gamma \to \Sigma$ the tube bundle whose classifying element $\alpha \in \pi_n \cT_\infty$ is the image of $\gamma \in \pi_nBO$ under the homomorphism $\pi_n BO \to \pi_n \cT_\infty$ induced by the rigid tube map. Then $\beta(W_\gamma,\fF)=(\Sigma, \fF_\gamma)$ for all framings $\fF$ of $\Sigma$. \end{lemma}

\begin{proof} 
A tube bundle whose classifying element is in the image of $\pi_n BO \to \pi_n \cT_\infty$ is up to stabilization the same thing as a bundle of rigid tubes, all of which admit a standard generating function whose gradient gives a fibrewise gradient vector field $X$ with one unique non-degenerate zero on each fibre. Then $X^{-1}(0)$ is a section of $W$, hence diffeomorphic to $\Sigma$. Without loss of generality we may assume that near the critical point $X$ is the linear map $(x_{E}, x^{\perp}_{E},z) \mapsto (x_{E} ,- x^{\perp}_E,z) $, where $(x_E,x_E^{\perp},z) \in E \oplus E^{\perp} \oplus \bR$. Hence $dX$ is given by the same formula. Hence in this case the framing $\fF_X$ is therefore by inspection the same as $\fF_\gamma$, so we conclude $\beta(W,\fF)=(X^{-1}(0),\fF_X)=(\Sigma,\fF_\gamma)$. \end{proof}


Even more trivial from the viewpoint of framed bordism are those tube bundles which come from pseudo-isotopy theory.

\begin{lemma}\label{prop:weird} Let $\Sigma$ be a homotopy sphere and $W \to \Sigma$ a tube bundle. If the classifying element $ \alpha \in \pi_n \cT_\infty$ is in the image of the map $\pi_n \cH_\infty \to \pi_n \cT_\infty$, then $\beta(W, \fF)=(\Sigma, \fF)$ in $\Omega^{fr}_n$ for all framings $\fF$ of $\Sigma$. \end{lemma}

\begin{proof} 
Any bundle of trivial h-cobordisms admits a fibrewise non-vanishing vector field pointing inwards at the lower end and outwards at the upper end, see Proposition 2.5 of \cite{EM} for a proof, though the argument goes back to Laudenbach and Douady \cite{L}. Hence a tube bundle in the image of $\pi_n \cH_\infty$ also admits a fibrewise vector field $X$ with one unique non-degenerate zero on each fibre. Hence as before $X^{-1}(0)$ is a section, and additionally the framing can now be taken to be constant. \end{proof}

We deduce a constraint on the framed bordism class of Legendrian homotopy spheres admitting generating functions on these special tube bundles.

\begin{proposition}\label{cor:weird} Let $\Sigma_0$ and $\Sigma_1$ be homotopy spheres and let $\Sigma_0 \subset J^1 \Sigma_1$ be a Legendrian embedding which admits a generating function $f:W \to \bR$ on a tube bundle $W \to \Sigma$. Suppose that the classifying element $ \alpha \in \pi_n \cT_\infty$ lies in the subgroup generated by the images of $\pi_n BO$ and $\pi_n \cH_\infty$. Then $\Sigma_0$ and $\Sigma_1$ are equal in $\Theta_n/bP^{n+1} \subset \text{coker}(J_n)$.
\end{proposition} 

\begin{proof}
Let $\fF^1$ be a framing of $\Sigma_1$ and let $\fF^0$ be the induced framing of $\Sigma_0$. we have $\beta(W, \fF^1) = (\Sigma_0, \fF^0_f)$ by Lemma \ref{lem:rep}. But we also have $\beta(W, \fF^1) = (\Sigma_1, \fF^1_\gamma)$ for some $\gamma \in \pi_nBO$ by Lemmas \ref{prop:BO} and \ref{prop:weird}. Hence in $\Theta_n/bP_{n+1} \subset \text{coker}(J_n)$ we deduce $[\Sigma_0]=(\Sigma_0, \fF^0_f) + \text{im}(J_n) = (\Sigma_1, \fF^1_\gamma) + \text{im}(J_n) = [\Sigma_1]$.
\end{proof}

Theorem \ref{thm:existence} gives no constraint on the possible class of the tube bundle in $\pi_n \cT_\infty$, therefore we see no obvious application of Proposition \ref{cor:weird} to the problem at hand. Instead, we will use the map $\cT_\infty \to BG$ to constrain the framed bordism classes which arise as $\beta(W,\fF)$ for $W \to \Sigma$ an arbitrary tube bundle and $\fF$ a framing of a homotopy sphere $\Sigma$.


\section{The Waldhausen derivative}\label{sec: wald der}

\subsection{The derivative on the space of tubes}

Take a homotopy class $ \alpha \in \pi_n \cT_\infty$, which can be represented by a tube bundle $W \to S^n$ where the fibres are tubes $T_z$ in a fixed $\bR^{N+1}$, parametrized by $z \in S^n$. We may assume that all the tubes are standard when $x_{N+1} \leq -1$, say, i.e. $T_z \cap \{ x_{N+1} \leq -1 \} =  \{ x_{N+1} \leq -1 \} $ for all $z \in S^n$. Then the outwards normal $n_z$ to $\partial T_z$ extends to a nonvanishing fibrewise vector field $X_z$ on $T_z \cap \{x_{N+1} \geq -1 \}$, equal to $X_z=\partial_{N+1}$ outside a compact set. The existence and uniqueness up to homotopy of this extension are guaranteed by stabilizing in $N$ as much as needed. 

Restricting $X_z$ to the hyperplane $\{x_{N+1}=-1\}$ gives a map $S^N \to S^N$ on each fibre, which has degree $\pm 1$. Hence we obtain a homotopy class $D \alpha \in \pi_n G$, where $G$ is the space of stable homotopy equivalences of spheres. 

\begin{definition} The map $D: \pi_n \cT_\infty \to \pi_n G$ is called the {\it Waldhausen derivative}. \end{definition}

Recall that there is a natural isomorphism $\pi_n G \simeq \pi_n^s$ under which the $J$-homomorphism $J_n:\pi_n O \to  \pi_n^s$ is the map on homotopy groups $\pi_n O \to \pi_n G$ induced by the map $O \to G$. So we may think of $D \alpha $ as an element of $\pi_n^s$ via this isomorphism if we wish to. Denote by $\fF^{0}$ the canonical framing of $S^n$ as the boundary of the unit disk in $\bR^{n+1}$.

\begin{lemma}\label{lem:thom} The Waldhausen derivative $D \alpha \in \pi_n^s$ of the classifying element $\alpha \in \pi_n \cT_\infty$ of a tube bundle $V \to S^n$ corresponds to $\beta(V, \fF^{0}) \in \Omega^{fr}_n$ under the Pontryagin-Thom isomorphism. \end{lemma}

\begin{proof}
This is a matter of unraveling definitions. Take a fibrewise vector field $X$ on the tube bundle classified by $\alpha$ which has no zeros on $\{x_{N+1} \geq -1\}$, is outwards pointing along $\partial T$ and is standard at infinity. Then the restriction of $X$ to $\{x_{N+1}=-1\}$ is the Waldhausen derivative $D \alpha$ and the zero set $(X^{-1}(0), \fF^{0}_X)$ is a representative for $\beta(V, \fF^{0})$. So we identify $\beta(V, \fF^{0})$ as the Thom-Pontryagin construction applied to (the suspension of) $D \alpha$. 
\end{proof}

\subsection{The addition formula}

Next we interpret the Waldhausen derivative from the viewpoint of an arbitrary homotopy sphere. We begin with a decomposition lemma.

\begin{lemma}\label{lem:div}
Let $W \to \Sigma$ be a tube bundle over a homotopy sphere $\Sigma$ and let $\rho: S^n \to \Sigma$ be a homotopy equivalence. Then $W$ is stably equivalent to the fibre connect sum of $\rho^*(W)$ with the trivial tube bundle over $\Sigma$. 
\end{lemma}

\begin{proof}
Without loss of generality we may assume that $\rho$ preserves orientation, that $\rho$ restricts to a diffeomorphism on the upper hemisphere $D_+$ of $S^n$ and that $\rho(S^n \setminus D_+) \subset \Sigma \setminus \rho(D_+)$. Let $D_1=\rho(D_+) \subset \Sigma$, a standard $n$-disk in $\Sigma$. 
Let $D_0 \subset D_1$ be a smaller concentric disk and let $D_2=\Sigma \setminus D_0$. Fix a trivialization $D_2 \times T$ of the restriction of $W$ to $D_2$, which is possible since $D_2$ is contractible. 
Note that the image by $\rho$ of the lower hemisphere $D_-$ of $S^n$ has image contained in $D_2$. 
Taking the composition of $\rho: S^n \to \Sigma$ with the classifying map $\Sigma \to \cT_\infty $ for $W$ we obtain a classifying map $S^n \to \cT_\infty$ for a tube bundle $V \to S^n$. It now follows that the fibre connect sum of $V$ with the trivial tube bundle over $\Sigma$ is stably equivalent to $W$ because stable equivalence classes of tube bundles are classified by $\pi_n \cT_\infty$ and both of these tube bundles have $\alpha$ as their classifying element.
\end{proof}

The decomposition lemma induces an addition formula in framed bordism:

\begin{lemma}\label{lem:trick} Let $\Sigma$ be a homotopy sphere and $W \to \Sigma$ a tube bundle with classifying element $\alpha \in \pi_n \cT_\infty$. For any framing $\fF$ of $\Sigma$ we have $\beta(W, \fF) =D \alpha + (\Sigma,\fF)$ in $\Omega^{fr}_n$. \end{lemma}

\begin{proof}
Let $W \to \Sigma$ be the tube bundle classified by $W$ as the fibre connect sum of a tube bundle $V$ on $S^n$ and the trivial bundle of $\Sigma$ as in Lemma \ref{lem:div}, so that $V$ is classified by the class $ \alpha \in \pi_n \cT_\infty$. We can deform the framing $\fF$ so that under the fibre connect sum it splits as a trivial framing on the $S^n$ component and the same framing $\fF$ on $\Sigma$. Moreover, we may compute $\beta(W,\fF)$ using a fibrewise vector field $X$ on $W$ constructed by fibre connect summing a fibrewise vector field $Y$ on $V$ which is standard over the attaching disk and the constant, standard vector field on the trivial tube bundle over $\Sigma$. We deduce $\beta(W, \fF)=\beta(V, \fF^{0}) + (\Sigma, \fF)$ in $\Omega^{fr}_n$, and $\beta(V, \fF^{0})=D \alpha$ by Lemma \ref{lem:thom}.
\end{proof}


We have achieved the following interpretation of the Waldhausen derivative.

\begin{proposition}\label{prop:bordism}Suppose $\Sigma_0, \Sigma_1$ are homotopy spheres and $\Sigma_0 \subset J^1 \Sigma_1$ a Legendrian embedding admitting a generating function on a tube bundle classified by $\alpha \in \pi_n \cT_\infty$. Then the difference between $\Sigma_0 $ and $\Sigma_1$ in $\Theta_n/bP_{n+1} \subset \text{coker}(J_n)$ is equal to the coset $D \alpha + \text{im}(J_n)$. \end{proposition}

\begin{proof} 
Let $\fF_1$ be a framing on $\Sigma_1$ and $\fF_0$ the induced framing on $\Sigma_0$. By Lemma \ref{lem:rep} we have $\beta(W,\fF_1)=(\Sigma_0, \fF_0)$ and by Lemma \ref{lem:trick} we have $\beta(W, \fF_1)=D \alpha + (\Sigma_1, \fF_1)$. Hence in $\Theta_n/bP_{n+1}$ we have $[\Sigma_0] = (\Sigma_0, \fF_0) + \text{im}(J_n) = D \alpha + (\Sigma_1, \fF_1) + \text{im}(J_n) = D\alpha + [\Sigma_1]$. \end{proof}

\subsection{Factoring the Waldhausen derivative}
The Waldhausen derivative can be upgraded to a map at the level of spaces $\cT_\infty \to G$, which is defined in the same way, see \cite{W}. Furthermore, it is shown in \cite{W} that the same definition can be made to work in the topological category and one may therefore factor the Waldhausen derivative as a composition $\cT_\infty \to BG \to G$. 

Further, the map $BG \to G$ in this factoring was identified by Waldhausen and Bökstedt \cite{W,BW}. We recall that $\pi_n G \simeq \pi_n^s$ and $\pi_n BG \simeq \pi^s_{n-1}$ for $n \geq 1 $.  

\begin{theorem}[Bökstedt, Waldhausen]\label{prop:factor} For $n \geq 3$ the map $BG \to G$ has its induced map on homotopy groups $\pi_n BG \to \pi_n G$ equal to the map $\pi_{n-1}^s \to \pi_n^s$ given by multiplication with the Hopf element $\eta \in \pi_1^s$. 
\end{theorem}

In fact, in \cite{BW} it is shown that the map $BG \to G$ is an $\eta$-map, from which Theorem \ref{prop:factor} follows immediately. Let us briefly explain the notion of an $\eta$-map, though this will not be needed in what follows. For $X=\Omega^2Y$ a 2-fold loopspace we may define a map $X \to \Omega X$, i.e. a map $\Omega^2Y \to \Omega^3Y$, by pre-composition with the Hopf fibration $S^3 \to S^2$. Such a map is called an $\eta$-map. The based space of stable spherical fibrations $X=\bZ \times BG$ is an infinite loopspace and there is a homotopy equivalence $i:\Omega B G \to G$, so we have an $\eta$-map $BG \to G$.

The relevant consequence for us is the following, where we return to thinking of the Waldhausen derivative as a map on homotopy groups $D:\pi_n \cT_\infty \to \pi_nG$.

\begin{corollary}\label{cor:eta}
Under the isomorphism $\pi_n G \simeq \pi_n^s$ we have $\text{im}(D) \subset \eta \cdot \pi_{n-1}^s$ for all $n \geq 1$.
\end{corollary}

\begin{proof}
The class $D \alpha$ is in the image of $\pi_n BG \to \pi_n G$ since the Waldhausen derivative factors as the composition $\cT_\infty \to BG \to G$. So the conclusion follows from Theorem \ref{prop:factor} when $n \geq 3$ and $\pi_n^s= \eta \cdot \pi_{n-1}^s$ for $n=1,2$.
\end{proof}

\begin{remark}
In \cite{AACK} we construct monoidal models for all the spaces and maps under considerations which allow the factoring $\cT_\infty \to BG \to G$ to be upgraded to a factoring at the level of maps of monoids. This point is crucial and yields the factoring $B(\cT,\mathcal{Q}) \to B(G/O) \to G/O$ which appears in Theorem \ref{thm:general}. No doubt the factoring $\cT_\infty \to BG \to G$ may be further upgraded to a factoring at the level of infinite loop maps, but we do not keep track of the necessary additional structure in \cite{AACK}. 
\end{remark}

\subsection{Legendrian homotopy spheres}

We are now ready to prove Theorem \ref{thm:structural}. The key result is the following:

\begin{theorem}\label{cor:leg}
Let $\Sigma_0$ and $\Sigma_1$ be homotopy spheres. Suppose that there exists a Legendrian embedding $\Sigma_0 \subset J^1\Sigma_1$ which admits a generating function of tube type $f:W \to \bR$  over some tube bundle $W \to M$. Then the difference between $\Sigma_0$ and $\Sigma_1$ in $\Theta_n/bP_{n+1} \subset \text{coker}(J_n)$ lies in the subgroup $\eta \cdot \pi_{n-1}^s + \text{im}(J_n)$.\end{theorem}

\begin{proof} Let $W \to \Sigma_1$ be classified by $\alpha \in \pi_n \cT_\infty$. By Proposition \ref{prop:bordism}, the difference between $\Sigma_0$ and $\Sigma_1$ in $\Theta_n / bP_{n+1} \subset \text{coker}(J_n)$ is equal to the coset $D \alpha + \text{im}(J_n)$, and by Corollary \ref{cor:eta} $D \alpha$ is a multiple of $\eta$. \end{proof}

To apply the result in the case of a Lagrangian homotopy sphere we only need to invoke the existence theorem for generating functions of tube type. 

\begin{proof}[Proof of Theorem \ref{thm:structural} ] Suppose that $\Sigma_0,\Sigma_1$ are homotopy spheres. If there exists a Lagrangian embedding $\Sigma_0 \subset T^*\Sigma_1$, then by Theorem \ref{thm:existence} it admits a generating function of tube type on some tube bundle. We apply Theorem \ref{cor:leg} to get the desired result. \end{proof}


\bibliographystyle{amsplain}
\bibliography{references}

@article{A,
  title={Framed bordism and {L}agrangian embeddings of exotic spheres},
  author={Abouzaid, Mohammed},
  journal={Annals of Mathematics},
  pages={71--185},
  year={2012},
  publisher={JSTOR}
}

@article{ACGK,
  title={Twisted generating functions and the nearby {L}agrangian conjecture},
  author={Abouzaid, Mohammed and Courte, Sylvain and Guillermou, St{\'e}phane and Kragh, Thomas},
  journal={Duke Mathematical Journal},
  volume={174},
  number={5},
  pages={949--1011},
  year={2025},
  publisher={Duke University Press}
}

@article{AACK,
  title={Normal invariant of nearby {L}agrangians via twisted derivative},
  author={Abouzaid, Mohammed and {\'A}lvarez-Gavela, Daniel and Courte, Sylvain and Kragh, Thomas},
  journal={arXiv preprint arXiv:2505.12515},
  year={2025},
  pages={}
}

@inproceedings{B,
  title={The rational homotopy type of ${\Omega}${W}h{D}iff},
  author={B{\"o}kstedt, Marcel},
  booktitle={Algebraic Topology Aarhus 1982: Proceedings of a conference held in Aarhus, Denmark, August 1--7, 1982},
  pages={25--37},
  year={2006},
  organization={Springer}
}

@misc{H,
  author = {Hatcher, Allen},
  title = {Pictures of stable homotopy groups of spheres.},
  howpublished = {https://pi.math.cornell.edu/~hatcher/stemfigs/stems.pdf},
}

@inproceedings{BW,
  title={The map {$BSG \to A(^\ast) \to QS^0$}},
  author={B{\"o}kstedt, Marcel and Waldhausen, Friedhelm},
  booktitle={Algebraic topology and algebraic K-theory: proceedings of a Conference October 24-28, 1983},
  year={1987}
}

@article{EKS,
  title={Lagrangian exotic spheres},
  author={Ekholm, Tobias and Kragh, Thomas and Smith, Ivan},
  journal={Journal of Topology and Analysis},
  volume={8},
  number={03},
  pages={375--397},
  year={2016},
  publisher={World Scientific}
}

@article{ES,
  title={Exact {L}agrangian immersions with a single double point},
  author={Ekholm, Tobias and Smith, Ivan},
  journal={Journal of the American Mathematical Society},
  volume={29},
  number={1},
  pages={1--59},
  year={2016}
}

@article{EM,
  title={Wrinkling of smooth mappings-{II} {W}rinkling of embeddings and K. Igusa’s theorem},
  author={Eliashberg, YM and Mishachev, NM},
  journal={Topology},
  volume={39},
  number={4},
  pages={711--732},
  year={2000},
  publisher={Elsevier}
}

@inproceedings{G,
  title={Formes g{\'e}n{\'e}ratrices d’immersions lagrangiennes dans un espace cotangent},
  author={Giroux, Emmanuel},
  booktitle={G{\'e}om{\'e}trie Symplectique et M{\'e}canique: Colloque International La Grande Motte, France, 23--28 Mai, 1988},
  pages={139--145},
  year={2006},
  organization={Springer}
}

@article{KA,
  title={Parametrized ring-spectra and the nearby {L}agrangian conjecture},
  author={Kragh, Thomas},
  journal={Geometry \& Topology},
  volume={17},
  number={2},
  pages={639--731},
  year={2013},
  publisher={Mathematical Sciences Publishers}
}

@article{K,
  title={Generating {F}unctions in $\mathbb{R}^{2n}$ and the {H}atcher-{W}aldhausen map},
  author={Kragh, Thomas},
  journal={arXiv preprint arXiv:1804.02557},
  year={2018},
  pages={}
}

@article{KM,
  title={Groups of homotopy spheres: I},
  author={Kervaire, Michel A and Milnor, John W},
  journal={Annals of mathematics},
  volume={77},
  number={3},
  pages={504--537},
  year={1963},
  publisher={JSTOR}
}

@article{L,
  title={Formes diff{\'e}rentielles de degr{\'e} 1 ferm{\'e}es non singuli{\`e}res: classes d'homotopie de leurs noyaux},
  author={Laudenbach, F},
  journal={Commentarii Mathematici Helvetici},
  volume={51},
  number={1},
  pages={447--464},
  year={1976},
  publisher={Springer}
}

@article{PS,
  title={Bordism of flow modules and exact {L}agrangians},
  author={Porcelli, Noah and Smith, Ivan},
  journal={arXiv preprint arXiv:2401.11766},
  year={2024},
  pages={}
}

@article{PS2,
  title={Bordism from quasi-isomorphism},
  author={Porcelli, Noah and Smith, Ivan},
  journal={arXiv preprint arXiv:2509.21587},
  year={2025},
  pages={}
}

@book{R,
  title={Complex cobordism and stable homotopy groups of spheres},
  author={Ravenel, Douglas C},
  year={2003},
  publisher={American Mathematical Soc.}
}

@incollection{W,
  title={Algebraic {K}-theory of spaces: a manifold approach},
  author={Waldhausen, Friedhelm},
  booktitle={Current trends in algebraic topology},
  volume={1},
  year={1982},
  publisher={}
}

\end{document}